\documentclass[journal]{IEEEtran}
\IEEEoverridecommandlockouts                              
\usepackage{graphicx}
\usepackage{subfigure}
\usepackage{epsfig}
\usepackage{color}
\usepackage{enumerate}
\usepackage{cite}
\newcommand{\subparagraph}{}
\usepackage{graphicx}
\usepackage{amssymb}
\usepackage{multirow}
\usepackage{amsthm}
\usepackage{amsmath}

\newcommand{\bea}{\begin{eqnarray}}
\newcommand{\eea}{\end{eqnarray}}
\newcommand{\beas}{\begin{eqnarray*}}
\newcommand{\eeas}{\end{eqnarray*}}
\newcommand{\leftm}{\left[\begin{array}}
\newcommand{\rightm}{\end{array}\right]}

\def\tr{{\mbox{\bf trace}}}

\def\diag{{\mbox{\bf diag}}}

\newtheorem{thm}{Theorem}

\newtheorem{prop}[thm]{Proposition}

\newtheorem{Assumption}[thm]{Assumption}
\newtheorem{Lemma}[thm]{Lemma}
\allowdisplaybreaks

\title{\LARGE \bf
An Iterative Method for Nonconvex Quadratically Constrained
Quadratic Programs}
\author{Chuangchuang Sun and Ran~Dai
\thanks{Chuangchuang Sun and Ran Dai are with the Aerospace Engineering Department, Iowa State University, Ames, IA. Emails:
        {\tt\small ccsun@iastate.edu and dairan@iastate.edu}}%
}

\begin{document}
\maketitle
\thispagestyle{empty}
\pagestyle{empty}


\begin{abstract}
This paper examines the nonconvex quadratically constrained quadratic programming (QCQP) problems using an iterative method.
One of the existing approaches for solving nonconvex QCQP problems relaxes the rank one constraint on the unknown matrix into semidefinite constraint to obtain the bound on the optimal value without finding the exact solution. By reconsidering the rank one matrix, an iterative rank minimization (IRM) method is proposed to gradually approach the rank one constraint. Each iteration of IRM is formulated as a convex problem with semidefinite constraints. An augmented Lagrangian method, named extended Uzawa algorithm, is developed to solve the subproblem at each iteration of IRM for improved scalability and computational efficiency. Simulation examples are presented using the proposed method and comparative results obtained from the other methods are provided and discussed.
\end{abstract}
\begin{IEEEkeywords}
Quadratically Constrained Quadratic Programming; Semidefinite Programming; Nonconvex Optimization; Augmented Lagrangian Method
\end{IEEEkeywords}

\section{INTRODUCTION}
The general/noncovex quadratically constrained quadratic programming (QCQP) problem has recently attracted significant interests due to its wide applications. For example, any polynomial problems of optimizing a polynomial objective function while satisfying polynomial inequality constraints can be reformulated as general QCQP problems~\cite{burer2012representing,d2003relaxations}. In addition, we can find QCQP applications in the areas of maxcut problems~\cite{4099492}, production planning~\cite{rutenberg1971product}, signal processing~\cite{luo2010semidefinite}, sensor network localizations~\cite{biswas2006semidefinite}, and optimal power flow~\cite{1664986,5971792}, just to name a few.

Convexification and relaxation techniques have been commonly used when solving nonconvex optimization problems~\cite{Acikmese2011341,5738669,6190765}. Efforts toward solving nonconvex QCQP problems have been pursued in two directions, obtaining a bound on the optimal value and finding a feasible solution. For simplicity, the QCQPs discussed below represent general/nonconvex QCQPs. Extensive relaxation methods have been investigated to obtain a bound on the optimal value of a QCQP. The linear relaxation approach introduces extra variables to transform the quadratic objective and constraints into bilinear terms, which is followed by linearization of the bilinears~\cite{al1995relaxation,qualizza2012linear}. The final linear formulation reaches a bound on the QCQP optimal value with fast convergence, but low accuracy.
The semidefinite programming (SDP) relaxation introduces a rank one matrix to replace the quadratic objective and constraints with linear matrices equalities/inequalities. However, the nonlinear rank one constraint on the unknown matrix is substituted by semidefinite relaxation. In general, the SDP relaxation reaches a tighter bound on the optimal value than that obtained from linear relaxation~\cite{d2003relaxations}.
A detailed discussion of various relaxation approaches and the comparison of their relative accuracy is provided in~\cite{Bao2011}.

However, finding a bound on the optimal value of QCQP does not imply generating an optimal solution, not even a feasible one. One of the efforts for obtaining a feasible solution utilizes an iterative linearization approach to gradually improve the objective value~\cite{d2003relaxations}. However, this method does not provide any guarantee of convergence. Another approach is to generate randomized samples and solve the QCQP on average of the distribution. However, the randomization approach does not apply to problems with equality constraints and the optimality is not guaranteed. Branch and bound (BNB) method has been frequently utilized to search for the optimal solution of nonconvex problems~\cite{linderoth2005simplicial,androulakis1995alphabb,dai2012optimal}. Although BNB can lead to global optimal solution, the searching procedure is time consuming, especially for large scale optimization problems. Recent work in \cite{sojoudi2014exactness} proposes that the structure of the QCQP problems can be changed based on graph theory to obtain a low-rank solution which greatly reduces the gap between the exact solution and the relaxed one. Furthermore, works in \cite{lasserre2001global} generates a series of SDPs to solve polynomial optimization problems, which is applicable to small scale QCQPs.

After reviewing the literature, we come to a conclusion that a more efficient approach is required to solve QCQP problems. In our previous work of \cite{dai_acc2014,7172240}, an iterative rank minimization (IRM) method has been proposed to solve homogeneous QCQPs. Inspired by the SDP relaxation, the IRM method focuses on finding the unknown rank one matrix by gradually minimizing the rank of the unknown matrix. This paper explores the problem to inhomogeneous QCQPs and focuses on proof of convergence to local optimum with a determined linear convergence rate based on the duality theory and the Karush-Kuhn-Tucker conditions. Each iteration of IRM is formulated as a convex problem with semidefinite constraints. To improve the scalability and computational efficiency in solving large scale convex optimization problems with semidefinite constraints, an extended Uzawa algorithm, based on the augmented Lagrangian method~\cite{conn1991globally,fortin2000augmented}, is developed to solve the subproblem at each iteration of IRM. And convergence to global optimality for the extended Uzawa algorithm is provided.

To our knowledge, there has been no existing approach for solving nonconvex QCQP with guaranteed convergence and a determined convergence rate to a local optimum while also satisfying all nonconvex constraints. The special contribution of this paper is a novel iterative approach to solve nonconvex QCQPs and proof of the linear convergence of the iterative approach. Furthermore, the proposed approach is accomplished by solving each iteration via a scalable and computationally-efficient extended Uzawa algorithm.

In the following, the QCQP formulation is introduced in \S II. The IRM method is discussed in \S III with linear convergence rate and local optimality proof. In \S IV, the extended Uzawa algorithm for solving large scale SDPs is introduced and its convergence analysis is discussed. Simulation examples are presented in \S V to verify the effectiveness of the proposed methods. We conclude the paper with a few remarks in \S VI.


\section{Problem Formulation}
A general homogeneous QCQP problem can be expressed in the form of
\begin{eqnarray}
\label{eq:nc_qcqp}
     & J=\min{\mathbf{x}^TQ_0\mathbf{x}}\\
s.t. & \mathbf{x}^TQ_j\mathbf{x}\leq c_j,\; \forall\; j=1,\ldots,m\nonumber,
\end{eqnarray}
where $\mathbf{x}\in\mathbb{R}^n$ is the unknown vector to be determined, $Q_j\in\mathbb{S}^{n},\,j=0,\ldots,m$, is an arbitrary symmetric matrix, and $c_j\in\mathbb{R},\,j=1,\ldots,m$. 
Inhomogeneous QCQPs with linear terms can be reformulated as homogeneous ones by defining an extended vector $\tilde{\mathbf{x}}=[\mathbf{x}^T,t]^T$ as well as a new quadratic constraint $t^2 = 1$~\cite{luo2010semidefinite}.
Since $Q_j,\,j=0,\ldots,m,$ are not necessarily positive semidefinite, problem in (\ref{eq:nc_qcqp}) is generally classified as nonconvex and NP-hard, requiring global search for its optimal solution. Without loss of generality, the following approach to solve nonconvex QCQP problems focuses on homogeneous QCQPs.

\section{An Iterative Approach for Nonconvex QCQPs}
\subsection{The Lower Bound on the Optimal Value of QCQPs}
In order to solve the nonconvex QCQP in (\ref{eq:nc_qcqp}), the semidefinite relaxation method is firstly introduced to find a tight lower bound on the optimal objective value.
By applying interior point method, the relaxed formulation can be solved via a SDP solver~\cite{boyed1996}.
After introducing a rank one positive semidefinite matrix $X=\mathbf{x}\mathbf{x}^T$ and substituting the rank one constraint
%
by a positive semidefinite constraint, $X\succeq\mathbf{x}\mathbf{x}^T$, the relaxed formulation is written as
\begin{eqnarray}
\label{eq:nc_sdp}
     & J=\min_X{\langle X,Q_0\rangle}\\
s.t. & \langle X,Q_j \rangle\leq c_j,\; \forall\; j=1,\ldots,m,\; X\in\mathbb{S}_+^{n},\nonumber
\end{eqnarray}
where $\mathbb{S}_+$ denotes a positive semidefinite matrix and `$\langle\cdot\rangle$' denotes the inner product of two matrices, i.e., $\langle A,B \rangle = \tr(A^TB)$. The semidefinite constraint relaxes the original formulation in (\ref{eq:nc_qcqp}), which generally yields a tighter lower bound on the optimal value of (\ref{eq:nc_qcqp}) than the one obtained from linearization relaxation technique~\cite{Bao2011}.
Therefore, by reformulating the problem of (\ref{eq:nc_qcqp}) in the form of (\ref{eq:nc_sdp}), we obtain lower bound on the optimal value of (\ref{eq:nc_qcqp}). However, the relaxation method will not yield optimal solution of the unknown variables $\mathbf{x}$. Compared to the original formulation in (\ref{eq:nc_sdp}), the only difference of the relaxation approach is that the rank one constraint on matrix $X$ is excluded.
In order to obtain the optimal solution of $\mathbf{x}$, we reconsider the rank one constraint on matrix $X$ and propose an IRM approach to gradually reach the constraint.

\subsection{Iterative Rank Minimization Approach}\label{sec:irm}
Satisfying rank one constraint, $X=\mathbf{x}\mathbf{x}^T$,  for a unknown matrix $X$ is computationally complicated. The direct method is to examine eigenvalues of the matrix. When only one eigenvalue is nonzero, it can be claimed that the rank of the matrix is one.
However, for a unknown matrix $X\in \mathbb{S}^{n}$, there is no straight way to examine its eigenvalues before it is determined.
Although heuristic search methods have been used to minimize the rank of symmetric or asymmetric matrix, they cannot guarantee that the rank of the final matrix is one~\cite{fazel2003log,mesbahi1997rank}.

Based on the fact that when the rank of a matrix is one, it has only one nonzero eigenvalue. Therefore, instead of making constraint on the rank, the focus is on constraining the eigenvalues of $X$ such that the $n-1$ smallest eigenvalues of $X$ are all zero. The eigenvalue constraints on matrices have been used for graph design~\cite{shafi2011graph} and are applied here for rank minimization. Before addressing the detailed IRM approach, necessary observations that will be used subsequently in the approach are provided first.

\begin{prop}
For a nonzero positive semidefinite matrix, $X\in \mathbb{S}^{n}_{+}$, it is a rank one matrix if and only if $rI_{n-1}-V^TXV\succeq 0$, where $r=0$, $I_{n-1}$ is an identity matrix with dimension $n-1$, and $V\in \mathbb{R}^{n\times (n-1)}$ are the eigenvectors corresponding to the $n-1$ smallest eigenvalues of $X$.
\end{prop}

\begin{proof}
Assume the eigenvalues(nonnegative) of $X$ are sorted in descending orders in the form of $[\lambda_n,\lambda_{n-1},\ldots,\lambda_1]$.
Since the Rayleigh quotient of an eigenvector is its associated eigenvalue, then $rI_{n-1}\!-\!V^TXV\!\!$ is a diagonal matrix with diagonal elements set as $[r-\lambda_{n-1}, r-\lambda_{n-2},\ldots,r-\lambda_1]$.
Therefore the $n-1$ smallest eigenvalues of $X$ are all zero if and only if $rI_{n-1}-V^TXV\succeq 0$ and $r=0$. Then $X$ is a rank one matrix.
\end{proof}
%

From the above discussion, we will substitute the rank one constraint, $X=\mathbf{x}\mathbf{x}^T$, by the semidefinite constraint,
$rI_{n-1}-V^TXV\succeq 0$,
where $r=0$ and $V\in\mathbb{R}^{n\times (n-1)}$ are the eigenvectors corresponding to the $n-1$ smallest eigenvalues of $X$.
However, before we solve $X$, we cannot obtain the exact $V$ matrix, thus an iterative method is proposed to gradually minimize the rank of $X$. At each step $k$, we will solve the following semidefinite programming problem formulated as
\bea
     & J=\min_{X_k,r_k}{\langle X,Q_0 \rangle}+w^kr_k\label{eq:irm}\\
s.t. & \langle X,Q_j \rangle\leq c_j,\; \forall\; j=1,\ldots,m\nonumber\\
&r_kI_{n-1}-V_{k-1}^TX_kV_{k-1}\succeq 0,\;X_k\succeq 0\nonumber
\eea
where $w^{k}>1$ is a weighting factor for $r_k$ in $k$th iteration and $V_{k-1}$ are the eigenvectors corresponding to the $n-1$ smallest eigenvalues of
$X$ solved at previous iteration $k-1$.
At each step, we are trying to optimize the original objective function and at the same time minimize the newly introduced parameter $r$ such that when $r=0$, the rank one constraint on $X$ is satisfied. Meanwhile, since $X_k$ is constrained to be positive semidefinite, the term $V_{k-1}^TX_kV_{k-1}$ is positive semidefinite as well, which implies that the value of $r_k$ is nonnegative in order to satisfy $r_kI_{n-1}-V_{k-1}^TX_kV_{k-1}\succeq 0$ in (\ref{eq:irm}). The above approach is repeated until $r_k\leq\epsilon$, where $\epsilon$ is a small threshold for stopping criteria.
Once the rank one matrix $X$ is obtained, the optimal solution of $x$ is determined by $\mathbf{x}=\sqrt{\lambda_{n}}\mathbf{v}$,
where $\lambda_{n}\in\mathbb{R}$ is the largest eigenvalue of $X$ and $\mathbf{v}\in\mathbb{R}^n$ is the corresponding eigenvector.
The IRM algorithm is summarized below.

\vspace{0.5cm}
\framebox{\begin{minipage}[c][1\totalheight][t]{0.9\linewidth}%
\textbf{Algorithm: Iterative Rank Minimization}\\
\textbf{Input: }{Parameters $Q_0,\,Q_j,\, c_j,\,w$,\, $j=1,\ldots,m$}\\
\textbf{Output: }{Unknown rank one matrix $X$ and unknown state vector $\mathbf{x}$ }\\
\textbf{begin}{\par}
\begin{enumerate}
\item \textbf{initialize }{Set $k=0$, solve (\ref{eq:nc_sdp}) to find $X_0$ and obtain $V_0$ from $X_0$, set $k=k+1$}{\par}
\item \textbf{while }{$r_k>\epsilon$}{\par}
\item {Solve problem (\ref{eq:irm}) and obtain $X_k$, $r_k$}{\par}
\item {Update $V_k$ from $X_k$}{\par}
\item {$k=k+1$}{\par}
\item {\textbf{end while}}{\par}
\item {Find $\mathbf{x}$ from $X$}{\par}
\end{enumerate}
\textbf{end}%
\end{minipage}}
%
\\

Different from the nonlinear programming (NLP) method~\cite{bertsekas1999nonlinear} which can be applied to solve nonconvex QCQP problems, the IRM algorithm above does not require initial guess of the unknown variables. Furthermore, except the newly introduced variable $r_k$, there are no additionally introduced unknown variables in the formulation. This simple procedure can be easily implemented for any nonconvex QCQP problems. 
Further analysis of the convergence properties of IRM with a determined linear rate is discussed below.

\subsection{Convergence of IRM Algorithm}
%
\begin{prop}
$\lim_{k \to +\infty}r_k=0$ with linear rate in the IRM algorithm if the problem formulated in (\ref{eq:irm}) is feasible.
\end{prop}
\begin{proof}
The Lagrangian of (\ref{eq:irm}) is constructed as
$\mathcal{L} =\langle X_k,Q_0 \rangle + w^{k}r_k+\sum_{j=1}^m \lambda_j(\langle X_k,Q_j \rangle - c_j)  -\langle S_1, r_k\textbf{I}_{n-1} - V^{T}_{k-1}X_kV_{k-1}\rangle -\langle S_2, X_k\rangle $,
where $\lambda_j\in\mathbb{R}> 0$, $j=1,\ldots,m$, $S_1 \in \mathbb{S}_+^{n-1}$ and $S_2 \in \mathbb{S}_+^{n}$ are the Lagrangian dual multipliers.
The dual function is then expressed as
$g_k(\lambda,\mu, S_1,S_2) = \inf_{X_k,r_k}\mathcal{L}(X_k,r_k,\lambda,S_1,S_2)$.
Consequently, the dual problem is built as,
\bea\label{eq:Dual problem}
     & J = \min_{\lambda,S_1,S_2} -\sum_{j=1}^m \lambda_jc_j \\
s.t. & Q_0 + \sum_{j=1}^m \lambda_jQ_j + V_{k-1}S_1V_{k-1}^{T} - S_2 = \textbf{0} \nonumber\\
     & w^k - \tr(S_1) = 0  \nonumber\\
     & \lambda_j \ge 0,\, S_1 \succeq \textbf{0},\, S_2 \succeq \textbf{0},\,\forall\, j=1,\ldots,m. \nonumber
\eea
It is obvious that the problem described in (\ref{eq:irm}) is convex. Moreover, it can be verified that the Slater's conditions are satisfied, which leads to the conclusion that the strong duality holds. Thus, at the optimal solution points of iteration $k$ and $k+1$, denoted as $(X_k^*,r_k^*,\lambda_k^*,(S_1^*)_{k},(S_2^*)_{k})$ and $(X_{k+1}^*,r_{k+1}^*,\lambda_{k+1}^*,(S_1^*)_{k+1},(S_2^*)_{k+1})$, the objective value of the primal problem equals to that of the dual problem, then $\langle X_k^{*},Q_0 \rangle  + w^kr_k^* = - \sum_{j=1}^m (\lambda_j^*)_kc_i = g_k$ and $
    \langle X_{k+1}^{*},Q_0 \rangle + w^{k+1}r_{k+1}^* = -\sum_{j=1}^m (\lambda_j^*)_{k+1}c_j =g_{k+1}$.
Since the optimization problem in (\ref{eq:irm}) is assumed to be feasible, the dual problem in (\ref{eq:Dual problem}) is bounded and matrix $X_k$ is finite. Then in the above two equations, subtracting one from the other yields $\lim_{k \to +\infty}(wr_{k+1}^* - r_k^*)= \lim_{k \to +\infty} \frac{(g_{k+1}-g_k) - \langle X_{k+1}^{*}- X_k^{*},Q_0 \rangle}{w^k} = 0,\, \forall\, w>1$.
Since $r^* = 0$, $r_{k+1}^* > 0$, $r_{k}^*>0$, and $0 < \frac{1}{w} < 1$, then $\lim_{k \to +\infty}\frac{|r_{k+1}^* - r^*|} {|r_{k}^* - r^*|}= \frac{1}{w}$. The above equation indicates that $r_k^*$ converges to $r^*$ linearly and $w^kr_k^*$ is non-increasing.
\end{proof}

\begin{prop}
$X_k$ converges to a local optimal solution $X^*$ in the IRM method.
\end{prop}
\begin{proof}
When $\lim_{k\to \infty}r_k= 0$, the semidefinite constraint in (\ref{eq:irm}) will become $\lim_{k \to \infty} -V_{k}^{T}X_{k+1}V_{k}\succeq \textbf{0}.$
For a semidefinite matrix $X_{k+1}\succeq \textbf{0}$, it leads to
$\lim_{k \to \infty} V_{k}^{T}X_{k+1}V_{k}=\textbf{0}$.
Since the above equation is a similarity transformation and as known before, $V_k \in \mathbb{R}^{n\times (n-1)}$ and $X_{k+1} \in \mathbb{S}_+^{n}$, it implies that the rank of $X_k$ is no more than one when $k$ approaches infinity. Hence we have
$\lim_{k \to \infty} V_{k}^{T}X_kV_{k} = \textbf{0}$.
Subtracting $\lim_{k \to \infty} V_{k}^{T}X_kV_{k}$ from $\lim_{k \to \infty} V_{k}^{T}X_{k+1}V_{k}$ yields $\lim_{k \to \infty} V_{k}^{T}(X_{k+1} - X_{k})V_{k} = \textbf{0}$. As $X_k$ is a positive semidefinite matrix with a rank of no more than one, one can get $\lim_{k \to \infty} X_{k+1} = \alpha X_k$, where $\alpha\in\mathbb{R}$.
%
The above relationship indicates that $V_{k}$ is constant when $k \to \infty$ since $X_k$ and $\alpha X_k$ share the same eigenvectors. Consequently, with a fixed $V_k$, the subproblem will converge to a local optimum.
\end{proof}

\section{An Extended Uzawa Algorithm for Convex Optimization}
Each iteration of IRM method requires to solve a convex optimization problem with linear matrix inequalities (LMIs). For large scale QCQPs, the performance of the convex optimization solver determines the scalability and computational capability of the proposed IRM method. However, existing convex optimization solvers, such as SeDuMi~\cite{sturm1999using} and SDPT3~\cite{toh1999sdpt3}, are time consuming and not applicable to large scale convex problems, especially to problems with a dimension larger than $100$ and multiple LMIs. Therefore, an extended Uzawa algorithm is developed here to solve the convex optimization problem at each iteration of IRM.

The Uzawa algorithm is originally introduced to solve concave problems~\cite{elman1994inexact}. When strong duality holds for a primal-dual problem, the optimal solution is the saddle point of the Lagrangian. Therefore, Uzawa algorithm is applied to iteratively approach the saddle point of the Lagrangian. Work in \cite{cai2010singular} has applied Uzawa algorithm to matrix completion problems with linear scalar/vector constraints. An extended Uzawa algorithm is developed here for convex problems with both scalar and LMI constraints.

To make it general, we consider the following convex optimization problem,
\begin{eqnarray}
\label{eq:SDP}
     & J=\min_x{f_0(\mathbf{x})}   \\
s.t. & f_i(\mathbf{x}) \le 0,\,i=1,...,N, \nonumber\\
     & \mathcal{A}_j(\mathbf{x}) \preceq \textbf{0},\, j=1,...,J, \nonumber
\end{eqnarray}
where $\mathbf{x}\in\mathbb{R}^n$ are the unknown variables, $f_i(\mathbf{x})\in\mathbb{R},\,i=0,...,N,$ are convex functions and $\mathcal{A}_j(\mathbf{x})\in\mathbb{S}^{n_j},\,j=1,...,J,$ are LMIs. For simplicity, we define that
$\mathcal{F}(\mathbf{x})\in \mathbb{R}^{l\times l}:=\diag(f_1(\mathbf{x}),...,f_N(\mathbf{x}),\mathcal{A}_1(\mathbf{x}),...,\mathcal{A}_J(\mathbf{x}))$. Without loss of generality, we assume that $\mathcal{F}(x)$ is symmetric.
In addition, as linear equality constraints can be written as a pair of linear inequality constraint, they are omitted in the formulation. The lagrangian function for (\ref{eq:SDP}) is given by
$\mathcal{L}(\mathbf{x},S)=f_0(\mathbf{x})+\langle S, \mathcal{F}(\mathbf{x})\rangle$, where $S\in \mathbb{R}^{l\times l}$ is the dual matrix variable. For a convex problem in (\ref{eq:SDP}) satisfying Slater's condition, the strong duality holds such that the primal-dual optimal pair, $(\mathbf{x}^*,S^*)$, has the relationship such that $\text{sup}_S \text{inf}_\mathbf{x} \mathcal{L}(\mathbf{x},S) = \mathcal{L}(\mathbf{x}^*,S^*) = \text{inf}_\mathbf{x} \text{sup}_S \mathcal{L}(\mathbf{x},S)$.

Consequently, the saddle point of the Lagrangian is the optimal pair, $(\mathbf{x}^*,S^*)$, which can be determined via the Uzawa's algorithm. As the initial value of Lagrangian multipliers is trivial for a convex problem, $S_0 = \textbf{0}$ is selected as the starting point and the iteration procedure of the extended Uzawa algorithm is formulated as
\begin{eqnarray}
\label{eq:Uzawa}
     & \mathbf{x}_h = \text{arg min}_\mathbf{x}\mathcal{L}(\mathbf{x},S_{h-1})  \\
     & S_h = [S_{h-1} + \delta_h \mathcal{F}(x_h)]_+,\nonumber
\end{eqnarray}
where $\delta_h > 0$ is the step size at iteration $h$ and the operator $[\bullet]_+$ is defined according to the data type of $'\bullet'$. For a vector $s\in\mathbb{R}^l$, $[s_i]_+ = \text{max}(0,s_i),\,i=1,\ldots,l$. While for a matrix $S\in \mathbb{S}^{l}$ with eigenvalues $\lambda\in \mathbb{R}^l$ and corresponding eigenvectors $V\in\mathbb{R}^{l\times l}$, $[S]_+ = V \diag(\text{max}(\lambda,0))V^T$, where $\text{max}(\lambda,0)$ will replace negative value elements in $\lambda$ with zeros. 
The following efforts focus on the convergence proof of the extended Uzawa algorithm stated in (\ref{eq:Uzawa}). The convergence proof of Uzawa algorithm for convex problems with linear scalar/vector constraints can be found in ~\cite{cai2010singular,cheng1984gradient}. However, the statement below focuses on proving the convergence of the extended Uzawa algorithm developed to solve more general convex problems including LMI constraints.

\begin{Assumption}\label{Assumption:Lipchitzs}
$\mathcal{F}(\mathbf{x})$ is Lipchitz continuous. Namely, for any $\mathbf{x}_1,\mathbf{x}_2$,
$\|\mathcal{F}(\mathbf{x}_1) - \mathcal{F}(\mathbf{x}_2)\|_F \le L(\mathcal{F})\|\mathbf{x}_1 -\mathbf{x}_2\|_F$ holds for some nonnegative constant $L(\mathcal{F})$.
\end{Assumption}


\begin{Assumption}\label{Assumption:xi}
	The objective function $f_0(\mathbf{x})$ is strongly convex.
\end{Assumption}

We first establish a preparatory lemma for the convergence proof.
\begin{Lemma}\label{Lemma:Positive}
Let$(\mathbf{x}^*,S^*)$ be an optimal primal-dual pair for (\ref{eq:SDP}), then for each $\delta \ge 0$, we have
$S^* = [S^* + \delta \mathcal{F}(\mathbf{x}^*)]_+$.
\end{Lemma}

\begin{proof}
As $(\mathbf{x}^*,S^*)$ is an optimal primal-dual pair, then $\langle S^*,\mathcal{F}(\mathbf{x}^*) \rangle = 0$, $S^* \succeq \textbf{0}$, and $\mathcal{F}(\mathbf{x}^*) \preceq \textbf{0}$.
Additionally, based on the fact that matrices $\mathcal{F}(\mathbf{x}^*)$ and $S^*$are symmetric, then $S^*\mathcal{F}(\mathbf{x}^*)$ is diagonalizable and can be written as $ S^*\mathcal{F}(\mathbf{x}^*)  = \mathcal{F}(\mathbf{x}^*)S^*= Q\Lambda_{S^*} \Lambda_{\mathcal{F}(\mathbf{x}^*)} Q^T = \textbf{0}$,
where $\Lambda_{S^*}\Lambda_{\mathcal{F}(\mathbf{x}^*)}  = \textbf{0} $,  $\Lambda_{\mathcal{F}(\mathbf{x}^*)} \preceq \textbf{0}$, and $ \Lambda_{S^*} \succeq \textbf{0}$.
Consequently, we can get $[S^* + \delta \mathcal{F}(\mathbf{x}^*)]_+ = Q[\Lambda_{S^*} + \delta \Lambda_{\mathcal{F}(\mathbf{x}^*)}]_+Q^T = S^*$.
When $\mathcal{F}(\mathbf{x}^*)$ is a scalar constraint, Lemma (\ref{Lemma:Positive}) degenerates to
$\lambda^* = [\lambda^* + \delta \mathcal{F}(\mathbf{x}^*)]_+$, where $\lambda^*\in\mathbb{R}$ is the dual variable of $\mathcal{F}(\mathbf{x}^*)$.
\end{proof}



\begin{prop}\label{Theorem:UzawaConvergence}
	Assuming problem (\ref{eq:SDP}) satisfies Assumptions (\ref{Assumption:Lipchitzs}) and (\ref{Assumption:xi})  and the step size $\delta_h$ in (\ref{eq:Uzawa}) satisfies $0 < \text{inf}\delta_h \le \text{sup}\delta_h < 2\xi/\|L(\mathcal{F})\|^2$, where $\xi$ and $L(\mathcal{F})$ are the parameters in the two aforementioned assumptions, respectively, then the sequence obtained from (\ref{eq:Uzawa}) will converge to the global optimum of the convex problem (\ref{eq:SDP}) when its strong duality holds.
\end{prop}

\begin{proof}
	At each iteration $h$ of the extended Uzawa algorithm defined in (\ref{eq:Uzawa}), the solution $\mathbf{x}_h$ minimizes $\mathcal{L}(\mathbf{x},S_{h-1})=f(\mathbf{x})+\langle S_{h-1},\mathcal{F}(\mathbf{x})\rangle$, then the first order optimality condition is expressed as
	\begin{eqnarray}\label{eq:2theorem}
	\nabla f(\mathbf{x}_h) +
	\begin{bmatrix}
	\langle S_{h-1},\frac{\partial{\mathcal{F}}}{\partial{(\mathbf{x}_h^{1})}} \rangle,
	\ldots,
	\langle S_{h-1},\frac{\partial{\mathcal{F}}}{\partial{(\mathbf{x}_h^{n})}} \rangle
	\end{bmatrix}^T
	\!\!\!=0,
	\end{eqnarray}
	where the superscript denotes the index of vector $\mathbf{x}$. 
	Moreover, as $f_i,\,i=1,...,N$, is a convex function and $\mathcal{A}_j,\,j=1,...,J$, is a LMI, the first order condition holds for the constructed $\mathcal{F}$ in the form of
	\begin{eqnarray}\label{eq:3theorem}
	\mathcal{F}(\mathbf{x}) - \mathcal{F}(\mathbf{x}_h) \succeq \sum\limits_{i=1}^n \frac{\partial{\mathcal{F}}}{\partial{(\mathbf{x}_h^{i})}}(\mathbf{x}^i - \mathbf{x}_h^{i}).
	\end{eqnarray}
	From (\ref{eq:2theorem}) and (\ref{eq:3theorem}), one can get
	\begin{eqnarray}\label{eq:4theorem}
	\langle \nabla f(\mathbf{x}_h),\mathbf{x}-\mathbf{x}_h \rangle + \langle S_{h-1},\mathcal{F}(\mathbf{x}) - \mathcal{F}(\mathbf{x}_h) \rangle \ge 0.
	\end{eqnarray}
	At the optimal point, the inequality relationship stated in (\ref{eq:4theorem}) will be satisfied as well when substituting $(\mathbf{x}_h,S_{h-1})$ by the primal-dual pair $(\mathbf{x}^*,S^*)$, which is expressed as
	\bea\label{eq:1theorem}
	\langle \nabla f_0(\mathbf{x}^*),\mathbf{x}-\mathbf{x}^* \rangle + \langle S^*,\mathcal{F}(\mathbf{\mathbf{x}}) - \mathcal{F}(\mathbf{x}^*) \rangle \ge 0.
	\eea
	Since the inequality relationships in (\ref{eq:4theorem}) and (\ref{eq:1theorem}) hold for all $\mathbf{x}$, substituting $\mathbf{x}$ in (\ref{eq:4theorem}) by $\mathbf{x}^*$ and (\ref{eq:1theorem}) by $\mathbf{x}_h$ and then adding them gives
	\begin{eqnarray} \label{eq:5theorem}
	\langle \nabla f(\mathbf{x}_h) - \nabla f(\mathbf{x}^*),\mathbf{x}_h - \mathbf{x}^* \rangle + \hspace{2.5cm} \nonumber\\
	\hspace{2cm} \langle S_{h-1} - S^*,\mathcal{F}(\mathbf{x}_h) - \mathcal{F}(\mathbf{x}^*) \rangle \le 0.
	\end{eqnarray}
	%
From Assumption \ref{Assumption:xi} and the properties of strongly convex function \cite{bertsekas2003convex}, there must exist a sufficiently small $\xi$ such that
	\begin{eqnarray}\label{eq:stronglyconvex}
	\langle \nabla f(\mathbf{x}_h) - \nabla f(\mathbf{x}^*),\mathbf{x}_h - \mathbf{x}^* \rangle    \ge \xi \|\mathbf{x}_h - \mathbf{x}^*\|_F^2.
	\end{eqnarray}
		Combing  (\ref{eq:5theorem})  and (\ref{eq:stronglyconvex}) gives $\langle S_{h-1} - S^*,\mathcal{F}(\mathbf{x}_h) - \mathcal{F}(\mathbf{x}^*) \rangle
	\le -\langle \nabla f(\mathbf{x}_h) - \nabla f(\mathbf{x}^*),\mathbf{x}_h - \mathbf{x}^* \rangle
	\le -\xi \|\mathbf{x}_h - \mathbf{x}^*\|_F^2$.
	Since $S_h$ in the extended Uzawa algorithm is updated via
	$S_h = [S_{h-1} + \delta_h \mathcal{F}(\mathbf{x}_h)]_+$, and Lemma (\ref{Lemma:Positive}) holds for $\delta_h$ such that $
	S^* = [S^* + \delta_h \mathcal{F}(\mathbf{x}^*)]_+$,
	then subtracting $S^*$ from $S_h$ gives $
	\|S_h - S^*\|_F^2  =\|[S_{h-1} + \delta_h \mathcal{F}(\mathbf{x}_h)]_+ - [S^* + \delta_h \mathcal{F}(\mathbf{x}^*)]_+ \|_F^2$.
	 The upper bound on $\|S_h - S^*\|_F^2$ is determined by
	 \bea\label{eq:bound}
	 \!\!\!\!&&\!\!\!\!\|S_h - S^*\|_F^2\le\| S_{h-1} - S^* + \delta_h(\mathcal{F}(\mathbf{x}_h) - \mathcal{F}(\mathbf{x}^*))\|_F^2  \nonumber\\
	 \!\!\!\!&\!\!\!\!\!\!\!\!\!\!\!\!\le\!\!\!\!\!\!\!\!\!\!\!\!& \!\!\!\!\| S_{h-1} - S^*\|_F^2 + 2\delta_h \langle S_{h-1} - S^*, \mathcal{F}(\mathbf{x}_h) - \mathcal{F}(\mathbf{x}^*) \rangle\nonumber\\
	 \!\!\!\!&&\!\!\!\!+ \delta_h^2 \|\mathcal{F}(\mathbf{x}_h) - \mathcal{F}(\mathbf{x}^*)\|_F^2 \nonumber\\
	 \!\!\!\!&\!\!\!\!\!\!\!\!\!\!\!\!\le\!\!\!\!\!\!\!\!\!\!\!\!&\!\!\!\!\| S_{h-1} - S^*\|_F^2-(2\xi\delta_h - \delta_h^2 L(\mathcal{F})^2)\|\mathbf{x}_h - \mathbf{x}^*\|_F^2.
	 \eea
	As it is assumed that $0 < \text{inf}\delta_h \le \text{sup}\delta_h < 2\xi/\|L(\mathcal{F})\|^2$, there exists a $\beta >0$ such that $2\xi\delta_h - \delta_h^2 L(\mathcal{F})^2 \ge \beta$ for all $h>1$. Then (\ref{eq:bound}) becomes $\|S_h - S^*\|_F^2 \le \| S_{h-1} - S^*\|_F^2 - \beta\|\mathbf{x}_h - \mathbf{x}^*\|_F^2$,
	and thus the proposition is proved.
\end{proof}

Considering the aforementioned subproblem, formulated in (\ref{eq:irm}), at each iteration of IRM, a quadratic term is included in the objective function and the proximal objective is expressed as
$J^{'} =\tau(X_k\bullet Q_0 + \omega^k r_k) + \frac{1}{2}(\|X_k\|_F^2 + r_k^2)$,
where $\tau$ is a weighting factor of the primary objective function. Optimizing the above proximal objective function can be handled as approximately optimizing the original objective at any prescribed accuracy as long as $\tau$ is properly selected and the problem is bounded. As stated in problem (\ref{eq:irm}), $X_k$ and $r_k$ are the unknown variables to be solved at iteration step $k$ of IRM. It is easy to check that $J^{'}$ is strongly convex and thus it satisfies Assumption \ref{Assumption:xi}.
Moreover, Assumption \ref{Assumption:Lipchitzs} and strong duality are satisfied as well due to the linear (matrix) constraints. According to the extended Uzawa algorithm, at each iteration step $h$ the unknown variables and Lagrangian multipliers in problem (\ref{eq:irm}) are updated through  (\ref{eq:Uzawa}). Noteworthily, the initial values of Lagrangian multipliers are set as $\lambda_i = 0$, $\mu_j = 0$, $S_1 = \textbf{0}$ and $S_2 = \textbf{0}$ as the initial setting is trivial for a convex problem.

\section{SIMULATION}
To verify the feasibility and efficiency of the proposed IRM method and the extended Uzawa Algorithm, two types of simulation examples are provided.  The first one solves mixed-boolean quadratic programming problems using the proposed IRM method where the subproblem at each iteration is solved via the extended Uzawa algorithm. The second one applies the IRM method in an optimal attitude control problem to verify the effectiveness of IRM in real applications. All of the simulation is run on a desktop computer with a 3.50 GHz processor and a 16 GB RAM.
	\subsection{Mixed-Boolean Quadratic Programming}
	In this subsection, the proposed IRM method is applied to solve mixed-boolean quadratic programming problems formulated as,
	\begin{eqnarray}
	\label{eq:miqcqp}
	& J=\min{\mathbf{x}^TQ_0\mathbf{x}}  \nonumber\\
	s.t. & \mathbf{x}^TQ_i\mathbf{x} + c_i \le 0 ,\; \forall\; i=1,\ldots,m, \\
	& \mathbf{x}_j \in \{1,-1\} ,\; \forall\; j \in N_I, \;\mathbf{x}_l \le \mathbf{x}_k\le \mathbf{x}_u ,\; \forall\; k \not\in N_I, \nonumber
	\end{eqnarray}
	where $\mathbf{x}\in \mathbb{R}^{50}$, $m$ is the number of inequality constraints, $N_I$ is the index set of the integer variables, and $\mathbf{x}_l$ and $\mathbf{x}_u$ are the lower and upper bounds of the continuous variables, respectively. The matrices $Q_0$ and $Q_i$, $i=1,\ldots,m$, are randomly generated and they are not necessarily positive semidefinite. Since the bivalent constraint on integer variables, $\mathbf{x}_j$, can be expressed as a quadratic equality constraint in the form of $(\mathbf{x}_j+1)(\mathbf{x}_j-1)=0,\,j\in N_I$, problem (\ref{eq:miqcqp}) can be converted to a nonconvex QCQP problem which can be solved by the proposed IRM method. The parameters in IRM are set as $w=2$ and $\epsilon=1e-5$.
	
	The comparative results are obtained from the Tomlab mixed-integer nonlinear programming solver, `minlpBB', which utilizes branch and bound to search for optimal solutions ~\cite{holmstrom1997tomlab}. 50 random cases are generated and solved via both IRM and `minlpBB'. For each case, the objective value obtained from both methods are recorded in Fig. \ref{f:CasesCompariosn50}. After comparison, the objective value obtained from IRM is always smaller than the corresponding one computed from `minlpBB' for all of the 50 cases. These facts validate the advantages of IRM in solving nonconvex QCQP problems. Furthermore, the value of $r_k$, representing the second largest eigenvalue of the unknown matrix $X$, at each iteration is demonstrated in Fig. \ref{fig:r_k} for one case. As $r_k$ converges to a number close to zero within $9$ iterations, Fig. \ref{fig:r_k} verifies the convergence of the IRM method to a rank one matrix. The other cases also yield zero $r_k$ at the convergence point. To save space, the values of $r_k$ for the other cases are not displayed here.
	
For each iteration of the IRM method, it will take the extended Uzawa algorithm 1 to 2 seconds to solve the convex problem formulated in (\ref{eq:irm}) for the 50 cases discussed above. Thus, the overall computation time using the combined IRM and extended Uzawa algorithms ranges from 10 to 20 seconds. However, it takes significantly increased computational time, around 10 to 100 times longer, to find the solution of each iteration of the IRM method using the `SeDuMi' convex optimization solver. Furthermore, the relative error, $|\frac{J(X^*_S) - J(X^*_U)}{J(X^*_S)}|$, between objective values from `SeDuMi', denoted by $J(X^*_S)$, and extended Uzawa algorithm, denoted by $J(X^*_U)$, averages 0.41\% for all 50 cases. The average computation time of the `minlpBB' solver is 2.1 secs for the 50 cases.
	\begin{figure}[!h]
		\vspace{-2.5cm}
		\centering
		\hspace{-6.7cm}
		\includegraphics[scale=0.43]{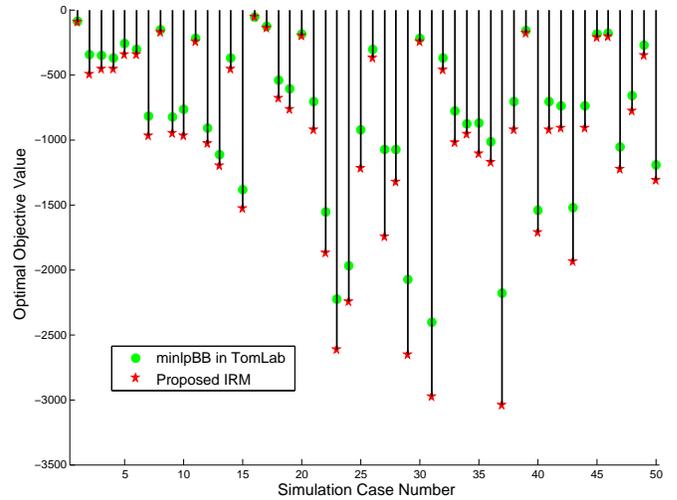}
		\hspace{-7.1cm}
		\vspace{-2.8cm}
		\caption{Comparative results between minlpBB and IRM for 50 cases}
		\vspace{-0.5cm}\label{f:CasesCompariosn50}
	\end{figure}
	
	\begin{figure}[!h]
		\vspace{-4.9cm}
		\centering
		\hspace{-6.7cm}
		\includegraphics[scale=0.56]{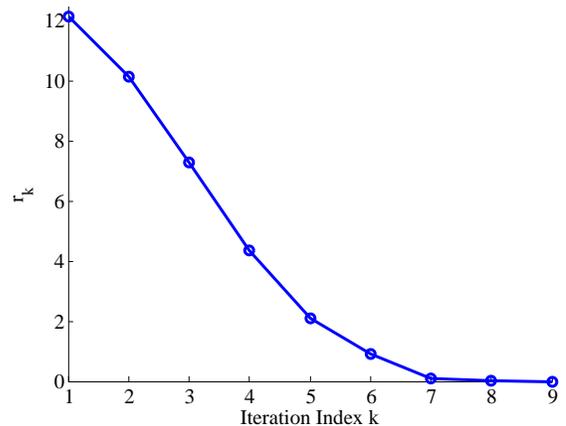}
		\hspace{-7.1cm}
		\vspace{-5.0cm}
		\caption{Value of $r_k$ at each iteration}
		\label{fig:r_k}
	\end{figure}
	
	\subsection{Optimal Attitude Control}
	To further verify the effectiveness and feasibility of IRM in real applications, the optimal attitude control problem for spacecraft is considered here. The objective of the optimal attitude control problem is to find the optimal control torque $\mathbf{u} \in \mathbb{R}^{3}$ to maneuver the orientation of spacecraft with minimum control efforts while satisfying a set of constraints over time interval $t \in [t_0,t_f]$. The constraints include boundary conditions, rotational dynamics, unit quaternion kinematics, and attitude forbidden and mandatory zones.

	In summary, the optimal control problem to minimize total control efforts for spacecraft reorientation with constraints can be formulated as
	\begin{eqnarray}\label{eq:performace index}
	&\min_{\mathbf{u},\omega,\mathbf q}  {\int_{t_0}^{t_f} \mathbf{u}^{T}\mathbf{u}\,dt}\\
	s.t. &\mathbf{J}\dot{\boldsymbol{\omega}}(t) = \mathbf{J}\boldsymbol{\omega} \times \boldsymbol{\omega} + \mathbf{u} \nonumber\\
	&\dot{\mathbf q}(t) = \frac{1}{2} \mathbf{\Omega}(t)\mathbf{q}(t),\;\|\mathbf q\| = 1 \nonumber\\
	&  \mathbf q^{T}M_{f_{l}} \mathbf q\le 0,\,l=1,\cdots,n'\nonumber \\
	&  \mathbf q^{T}M_{m_{s}} \mathbf q\ge 0 ,\,s=1,\cdots,h'\nonumber \\
	&|\mathbf{u}_i| \le \beta_{u_i},\, |\boldsymbol{\omega}_i| \le \beta_{\omega_i},\,i=1,2,3 \nonumber \\
	&\boldsymbol{\omega}(t_0) = \boldsymbol{\omega}_0,\, \boldsymbol{\omega}(t_f) = \boldsymbol{\omega}_f,\, \mathbf{q}(t_0) = \mathbf{q}_0,\, \mathbf{q}(t_f) = \mathbf{q}_f,\nonumber
	\end{eqnarray}
	where $\mathbf{J}=\diag(J_1,J_2,J_3)$ represents the moment of inertia matrix of the spacecraft in the body frame, $\boldsymbol{\omega} = [\omega_1,\omega_2,\omega_3]^{T} \in \mathbb{R}^{3}$ denotes the spacecraft angular velocity in the body frame, and $\mathbf{q} = [q_1,q_2,q_3,q_4]^{T}\in \mathbb{R}^{4},\,\|\mathbf q\| = 1$ denotes the attitude in unit quaternions. 
	In addition, $\boldsymbol{\omega}_0$, $\mathbf{q}_0$, $\boldsymbol{\omega}_f$, and $\mathbf{q}_f$ represent boundary conditions on angular velocity and attitude orientation. $n'$ and $h'$ represent the number of forbidden and mandatory zones, respectively. $\beta_{u_i}$ and $\beta_{\omega_i}$, $i=1,2,3$ denote upper bounds of the control torque and the angular velocity elements, respectively. The forbidden zone matrices, $M_{f_{l}}\in \mathbb{R}^{4\times 4}$, $l=1,\cdots,n'$ are determined by the vector, $\mathbf{x}_{F_l}\in \mathbb{R}^{3}$, to be avoided with a constrained angle $\theta_{F_l}$. The mandatory zone matrices, $M_{f_{s}}\in \mathbb{R}^{4\times 4}$, $s=1,\cdots,h'$, are determined similarly by vector $\mathbf{x}_{M_s}\in \mathbb{R}^{3}$ and angle $\theta_{M_s}$. More detailed description of the optimal attitude control problem can be referred to \cite{Yoonsoo2010}.
	
	By utilizing the discretization technique, the optimal attitude control problem formulated as a NLP problem in (\ref{eq:performace index}) can be transformed into a nonconvex QCQP problem in the form of (\ref{eq:nc_qcqp}). One simulation result is demonstrated here to reorient the spacecraft with minimum total control efforts while preventing its telescope pointing vector from the three forbidden zones and keeping the antenna vector in the mandatory zone within $t \in [0,20]$ seconds. Three forbidden zones are randomly selected without overlapping each other but may overlap with the mandatory zone. In addition, both initial and terminal attitude are properly selected to prevent violation of the attitude constraints. The spacecraft is assumed to carry a light-sensitive telescope with a fixed boresight vector $\mathbf{y}_t$, defined as $\mathbf{y}_t = [0,1,0]^{T}$, while the boresight vector of the antenna is set as $\mathbf{y}_a = [0,0,1]^{T}$, both in the spacecraft body frame. The other simulation parameters are given in Table \ref{table:Simulation parameters 2}.
	
	\begin{table}[!h]
		\caption{Simulation parameters for optimal attitude control problem} 
		\centering 
		\begin{tabular}{c c} 
			\hline\hline 
			Parameter & Value\\ [0.5ex] 
			\hline 
			J & diag[54,63,59]  \textit{kg$\cdot m^2$} \\ [1ex]
			\hline
			$|\omega_i|,i=1,2,3$ & $\le$ 0.3 rad/s    \\ [1ex] 
			\hline
			$|u_i|,i=1,2,3$ & $\le$ 3 rad/$s^2$    \\ [1ex]
			\hline
			Initial Attitude $\mathbf{q}_{0}$ & [0.82,0.52,-0.12,-0.23]$^T$    \\ [1ex]
			\hline
			Terminal Attitude $\mathbf{q}_{f}$ & [0.275386,-0.51,-0.78,-0.24]$^T$    \\ [1ex]
			\hline
			Mandatory zone 1  & $\mathbf{x}_{M_1}=[-0.81, 0.55, -0.19]^T$, $\theta_{M_1} = 70^{\circ}$ \\ [1ex]
			\hline
			Forbidden zone 1  & $\mathbf{x}_{F_1}=[0, -1, 0]^T$, $\theta_{F_1} = 40^{\circ}$\\ [1ex]
			\hline
			Forbidden zone 2  & $\mathbf{x}_{F_2}=[0, 0.82, 0.57]^T$, $\theta_{F_2} = 30^{\circ}$ \\ [1ex]
			\hline
			Forbidden zone 3  & $\mathbf{x}_{F_3}=[-0.12, -0.14, -0.98]^T$, $\theta_{F_3} = 20^{\circ}$ \\ [1ex]
			\hline \hline
		\end{tabular}
		\label{table:Simulation parameters 2} 
	\end{table}
	
	Figure \ref{f:3D_case2} presents the trajectories of the telescope pointing vector and the antenna pointing vector in the constrained three-dimensional (3D) space. 
	The value of $r_k$ at each iteration is provided in Fig. \ref{f:lambda2_case2}, which indicates that $r_k$ converges to zero within a few iterations. Furthermore, we find comparative results solved via the commercial NLP solver, SNOPT~\cite{holmstrom1997tomlab}. Depending on the initial guess of the unknown variables, the NLP solver cannot guarantee a convergent solution. When a group of initial guess is randomly generated, the convergent solutions from NLP solver lead to two sets of objective value, $31.79$ and $76.92$. However, the objective value found from IRM is $23.72$, which reduces $25.39\%$ compared to the smallest objective value obtained from NLP solver. This simulation example verifies the feasibility of implementing IRM in a real optimal control problem. It takes the IRM $743.6$ secs to generate the optimal solution while the NLP solvers takes a average of $8.4$ secs for convergent cases.
	\begin{figure}[!h]
		\begin{center}
			\vspace{-0.1cm}\scalebox{0.85}{\includegraphics[width=95mm]{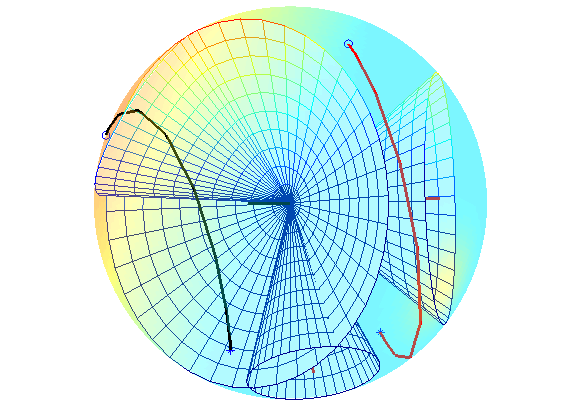}}
		\end{center}
		\vspace{-0.1cm}\caption{Trajectory of the telescope pointing vector (red) and the antenna point vector (black) in 3D space. The blue star and blue circle represent the initial and terminal orientations, respectively. The cone with black boresight vector represents the mandatory zone and the other three are the forbidden zones}\vspace{-0.1cm}\label{f:3D_case2}
	\end{figure}

	\begin{figure}[!h]
		\vspace{-5.0cm}
		\centering
		\hspace{-3.8cm}\scalebox{1.3}{\includegraphics[width=95mm]{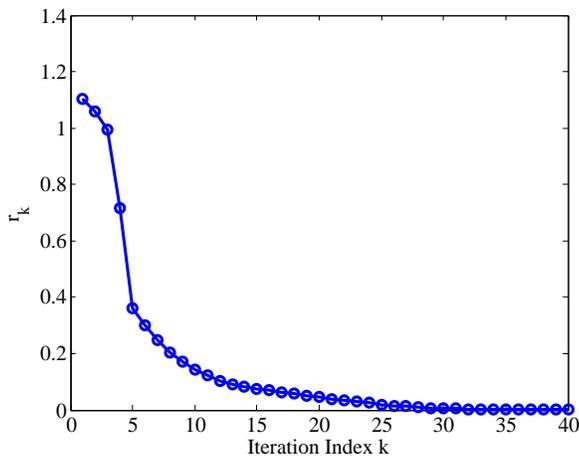}}
		\hspace{-3.8cm}
		\vspace{-5.2cm}\caption{Value of $r_k$ at each iteration}\vspace{-0.1cm}\label{f:lambda2_case2}
	\end{figure}

\section{CONCLUSIONS}
This paper proposes an Iterative Rank Minimization (IRM) method to solve nonconvex quadratically constrained quadratic programming (QCQP) problems. The subproblem at each iteration of IRM is formulated as a semidefinite programming (SDP) problem and an extended Uzawa algorithm, based on the augmented Lagrangian method, is developed to improve the scalability and computational efficiency in solving large scale SDPs at each iteration of IRM. Theoretical analysis on convergence of the proposed IRM method and the extended Uzawa algorithm is discussed. The effectiveness and improved performance of the proposed approach is verified by different types of simulation examples.


\bibliographystyle{IEEEtran}
\bibliographystyle{ieeeconf}
\bibliography{reflib}

\begin{thebibliography}{10}
\providecommand{\url}[1]{#1}
\csname url@samestyle\endcsname
\providecommand{\newblock}{\relax}
\providecommand{\bibinfo}[2]{#2}
\providecommand{\BIBentrySTDinterwordspacing}{\spaceskip=0pt\relax}
\providecommand{\BIBentryALTinterwordstretchfactor}{4}
\providecommand{\BIBentryALTinterwordspacing}{\spaceskip=\fontdimen2\font plus
\BIBentryALTinterwordstretchfactor\fontdimen3\font minus
  \fontdimen4\font\relax}
\providecommand{\BIBforeignlanguage}[2]{{%
\expandafter\ifx\csname l@#1\endcsname\relax
\typeout{** WARNING: IEEEtran.bst: No hyphenation pattern has been}%
\typeout{** loaded for the language `#1'. Using the pattern for}%
\typeout{** the default language instead.}%
\else
\language=\csname l@#1\endcsname
\fi
#2}}
\providecommand{\BIBdecl}{\relax}
\BIBdecl

\bibitem{burer2012representing}
S.~Burer and H.~Dong, ``Representing quadratically constrained quadratic
  programs as generalized copositive programs,'' \emph{Operations Research
  Letters}, vol.~40, no.~3, pp. 203--206, 2012.

\bibitem{d2003relaxations}
A.~d'Aspremont and S.~Boyd, ``Relaxations and randomized methods for nonconvex
  qcqps,'' \emph{EE392o Class Notes, Stanford University}, 2003.

\bibitem{4099492}
M.~Diehl, ``Formulation of closed-loop min-max mpc as a quadratically
  constrained quadratic program,'' \emph{Automatic Control, IEEE Transactions
  on}, vol.~52, no.~2, pp. 339--343, Feb 2007.

\bibitem{rutenberg1971product}
D.~P. Rutenberg and T.~L. Shaftel, ``Product design: Subassemblies for multiple
  markets,'' \emph{Management Science}, vol.~18, no. 4-part-i, pp. B--220,
  1971.

\bibitem{luo2010semidefinite}
Z.-q. Luo, W.-k. Ma, A.-C. So, Y.~Ye, and S.~Zhang, ``Semidefinite relaxation
  of quadratic optimization problems,'' \emph{Signal Processing Magazine,
  IEEE}, vol.~27, no.~3, pp. 20--34, 2010.

\bibitem{biswas2006semidefinite}
P.~Biswas, T.-C. Lian, T.-C. Wang, and Y.~Ye, ``Semidefinite programming based
  algorithms for sensor network localization,'' \emph{ACM Transactions on
  Sensor Networks (TOSN)}, vol.~2, no.~2, pp. 188--220, 2006.

\bibitem{1664986}
R.~Jabr, ``Radial distribution load flow using conic programming,'' \emph{Power
  Systems, IEEE Transactions on}, vol.~21, no.~3, pp. 1458--1459, Aug 2006.

\bibitem{5971792}
J.~Lavaei and S.~Low, ``Zero duality gap in optimal power flow problem,''
  \emph{Power Systems, IEEE Transactions on}, vol.~27, no.~1, pp. 92--107, Feb
  2012.

\bibitem{Acikmese2011341}
B.~Acikmese and L.~Blackmore, ``Lossless convexification of a class of optimal
  control problems with non-convex control constraints,'' \emph{Automatica},
  vol.~47, no.~2, pp. 341--347, 2011.

\bibitem{5738669}
S.~Narasimhan and R.~Rengaswamy, ``Plant friendly input design: Convex
  relaxation and quality,'' \emph{Automatic Control, IEEE Transactions on},
  vol.~56, no.~6, pp. 1467--1472, June 2011.

\bibitem{6190765}
E.~Elhamifar and R.~Vidal, ``Block-sparse recovery via convex optimization,''
  \emph{Signal Processing, IEEE Transactions on}, vol.~60, no.~8, pp.
  4094--4107, Aug 2012.

\bibitem{al1995relaxation}
F.~A. Al-Khayyal, C.~Larsen, and T.~Van~Voorhis, ``A relaxation method for
  nonconvex quadratically constrained quadratic programs,'' \emph{Journal of
  Global Optimization}, vol.~6, no.~3, pp. 215--230, 1995.

\bibitem{qualizza2012linear}
A.~Qualizza, P.~Belotti, and F.~Margot, \emph{Linear programming relaxations of
  quadratically constrained quadratic programs}.\hskip 1em plus 0.5em minus
  0.4em\relax Springer, 2012.

\bibitem{Bao2011}
X.~Bao, N.~Sahinidis, and M.~Tawarmalani, ``Semidefinite relaxations for
  quadratically constrained quadratic programming: A review and comparisons,''
  \emph{Mathematical Programming}, vol. 129, pp. 129--157, 2011.

\bibitem{linderoth2005simplicial}
J.~Linderoth, ``A simplicial branch-and-bound algorithm for solving
  quadratically constrained quadratic programs,'' \emph{Mathematical
  Programming}, vol. 103, no.~2, pp. 251--282, 2005.

\bibitem{androulakis1995alphabb}
I.~Androulakis, C.~Maranas, and C.~Floudas, ``$\alpha$bb: A global optimization
  method for general constrained nonconvex problems,'' \emph{Journal of Global
  Optimization}, vol.~7, no.~4, pp. 337--363, 1995.

\bibitem{dai2012optimal}
R.~Dai, U.~Lee, S.~Hosseini, and M.~Mesbahi, ``Optimal path planning for
  solar-powered uavs based on unit quaternions,'' in \emph{Decision and
  Control, 2012 IEEE 51st Annual Conference on}, pp. 3104--3109.

\bibitem{sojoudi2014exactness}
S.~Sojoudi and J.~Lavaei, ``Exactness of semidefinite relaxations for nonlinear
  optimization problems with underlying graph structure,'' \emph{SIAM Journal
  on Optimization}, vol.~24, no.~4, pp. 1746--1778, 2014.

\bibitem{lasserre2001global}
J.~B. Lasserre, ``Global optimization with polynomials and the problem of
  moments,'' \emph{SIAM Journal on Optimization}, vol.~11, no.~3, pp. 796--817,
  2001.

\bibitem{dai_acc2014}
R.~Dai, ``Three-dimensional aircraft path planning based on nonconvex quadratic
  optimization,'' in \emph{American Control Conference}, 2014, pp. 4561--4566.

\bibitem{7172240}
C.~Sun and R.~Dai, ``Identification of network topology via quadratic
  optimization,'' in \emph{American Control Conference}, July 2015.

\bibitem{conn1991globally}
A.~R. Conn, N.~I. Gould, and P.~Toint, ``A globally convergent augmented
  lagrangian algorithm for optimization with general constraints and simple
  bounds,'' \emph{SIAM Journal on Numerical Analysis}, vol.~28, no.~2, pp.
  545--572, 1991.

\bibitem{fortin2000augmented}
M.~Fortin and R.~Glowinski, \emph{Augmented Lagrangian methods: applications to
  the numerical solution of boundary-value problems}.\hskip 1em plus 0.5em
  minus 0.4em\relax Elsevier, 2000.

\bibitem{boyed1996}
L.~Vandenberghe and S.~Boyd, ``Semidefinite programming,'' \emph{SIAM Review},
  vol.~38, pp. 49--95, 1996.

\bibitem{fazel2003log}
M.~Fazel, H.~Hindi, and S.~P. Boyd, ``Log-det heuristic for matrix rank
  minimization with applications to hankel and euclidean distance matrices,''
  in \emph{American Control Conference}, vol.~3, 2003, pp. 2156--2162.

\bibitem{mesbahi1997rank}
M.~Mesbahi and G.~P. Papavassilopoulos, ``On the rank minimization problem over
  a positive semidefinite linear matrix inequality,'' \emph{Automatic Control,
  IEEE Transactions on}, vol.~42, no.~2, pp. 239--243, 1997.

\bibitem{shafi2011graph}
S.~Y. Shafi, M.~Arcak, and L.~El~Ghaoui, ``Graph weight design for laplacian
  eigenvalue constraints with multi-agent systems applications,'' in
  \emph{Decision and Control and European Control Conference, 50th IEEE
  Conference on}, 2011, pp. 5541--5546.

\bibitem{bertsekas1999nonlinear}
D.~P. Bertsekas, \emph{Nonlinear programming}.\hskip 1em plus 0.5em minus
  0.4em\relax Athena scientific, 1999.

\bibitem{sturm1999using}
J.~F. Sturm, ``Using sedumi 1.02, a matlab toolbox for optimization over
  symmetric cones,'' \emph{Optimization methods and software}, vol.~11, no.
  1-4, pp. 625--653, 1999.

\bibitem{toh1999sdpt3}
K.-C. Toh, M.~J. Todd, and R.~H. T{\"u}t{\"u}nc{\"u}, ``Sdpt3-a matlab software
  package for semidefinite programming, version 1.3,'' \emph{Optimization
  methods and software}, vol.~11, no. 1-4, pp. 545--581, 1999.

\bibitem{elman1994inexact}
H.~C. Elman and G.~H. Golub, ``Inexact and preconditioned uzawa algorithms for
  saddle point problems,'' \emph{SIAM Journal on Numerical Analysis}, vol.~31,
  no.~6, pp. 1645--1661, 1994.

\bibitem{cai2010singular}
J.-F. Cai, E.~J. Cand{\`e}s, and Z.~Shen, ``A singular value thresholding
  algorithm for matrix completion,'' \emph{SIAM Journal on Optimization},
  vol.~20, no.~4, pp. 1956--1982, 2010.

\bibitem{cheng1984gradient}
Y.~Cheng, ``On the gradient-projection method for solving the nonsymmetric
  linear complementarity problem,'' \emph{Journal of Optimization Theory and
  Applications}, vol.~43, no.~4, pp. 527--541, 1984.

\bibitem{bertsekas2003convex}
D.~Bertsekas and A.~Nedic, \emph{Convex analysis and optimization
  (conservative)}.\hskip 1em plus 0.5em minus 0.4em\relax Athena Scientific,
  2003.

\bibitem{holmstrom1997tomlab}
K.~Holmstr{\"o}m, ``Tomlab--an environment for solving optimization problems in
  matlab,'' in \emph{Proceedings for the Nordic Matlab Conference'97}, 1997.

\bibitem{Yoonsoo2010}
Y.~Kim, M.~Mesbahi, G.~Singh, and F.~Y. Hadaegh, ``On the convex
  parameterization of constrained spacecraft reorientation,'' \emph{Aerospace
  and Electronic Systems, IEEE Transactions on}, vol.~46, no.~3, pp. 1097
  --1109, july 2010.

\end{thebibliography}
\end{document}